\documentclass[12pt]{article}
\usepackage{amsmath}
\usepackage{amsthm}
\usepackage{amssymb}
\usepackage[hidelinks]{hyperref}
\usepackage{xcolor}
\usepackage{graphicx}
\usepackage[a4paper,  margin=2.5cm]{geometry}

\newtheorem{de}{Definition}

\newtheorem{pr}[de]{Proposition}

\newtheorem{re}[de]{Remark}
\newtheorem{cor}[de]{Corollary}
\newtheorem{con}[de]{Conjecture}
\theoremstyle{definition}

\newcommand{\dif}{\mathrm{d}}

\newcommand{\norm}[1]{\Vert #1 \Vert}  
\newcommand{\abs}[1]{\vert #1 \vert}

\newcommand{\DD}{\mathcal{D}} 
\newcommand{\RR}{\mathbb{R}}   
 
\newcommand{\NN}{\mathbb{N}}
\newcommand{\Ss}{\mathbb{S}}

\DeclareMathOperator{\di}{div}

\DeclareMathOperator{\curl}{curl}

\DeclareMathOperator{\grad}{grad}

\newcommand{\R}{{\mathbb{R}}}
\newcommand{\Z}{{\mathbb{Z}}}

\newcommand{\beq}{\begin{equation}}
\newcommand{\eeq}{\end{equation}}
\newcommand{\bea}{\begin{eqnarray}}
\newcommand{\eea}{\end{eqnarray}}
\newcommand{\ben}{\begin{eqnarray*}}
\newcommand{\een}{\end{eqnarray*}}
\newcommand{\bem}{\begin{enumerate}}
\newcommand{\eem}{\end{enumerate}}
\newcommand{\ra}{\rightarrow}

\newcommand{\wt}{\widetilde}
\newcommand{\wh}{\widehat}
\newcommand{\less}{\backslash}

\def \d{\mathrm{d}}

\newcommand{\ip}[1]{\langle #1 \rangle}
\newcommand{\ignore}[1]{}

\newcommand{\vol}{{\rm vol}}

\newcommand{\vphi}{\varphi}

\newcommand{\nvec}{\mbox{\boldmath{$n$}}}

\renewcommand{\div}{{\rm div}\, }

\begin{document}

\title{\textbf{BPS Skyrme models and contact geometry}}

\author{R. Slobodeanu \thanks{Faculty of Physics, University of Bucharest, 405 Atomi\c stilor, POB MG-11,
Bucharest-M\u agurele, Romania. E-mail: \texttt{radualexandru.slobodeanu@g.unibuc.ro}}
	\and 
	J.M. Speight \thanks{School of Mathematics, University of Leeds,
	Leeds LS2 9JT, England. E-mail: \texttt{J.M.Speight@leeds.ac.uk}}
	}

\date{\today}



\maketitle

\begin{abstract}
\noindent 
A Skyrme type energy functional for maps $\vphi:M\ra N$ from an oriented Riemannian 3-manifold $M$ to a 
contact 3-manifold $N$ is defined, generalizing the BPS Skyrme energy of Ferreira and Zakrzewski. This energy has a topological lower bound, attained by solutions of a first order self-duality equation which we call (strong) Beltrami maps. In the case where $N$ is the 3-sphere, we show that the original Ferreira-Zakrzewski model (which has $N=\Ss^3$ with the standard contact structure) can have no BPS solutions on $M=\Ss^3$ with $|\deg(\vphi)|>1$ if the coupling constant has the lowest admissible value.
\end{abstract}

\section{Introduction}

The Skyrme model \cite{sky} is an approximate low energy effective theory for nuclear physics, where baryons are modelled as topological solitons of pion fields.
More precisely a baryon is represented by an energy minimising, topologically nontrivial map (called a \textit{skyrmion}) $\varphi: \mathbb{R}^3 \to \mathbb{S}^3$ with $\lim_{|x| \to \infty}\vphi(x)=\vphi_0$, some constant point in $\Ss^3$. The degree of such a map (a topological invariant) is identified with the \textit{baryon number}. The \textit{Skyrme energy functional} is 
\begin{align*}
E_{\texttt{Skyrme}}(\varphi)&=\frac{1}{2}\int \big\{c\abs{\dif \varphi}^2+
\frac{4}{c} \abs{\dif\varphi \wedge \dif \varphi}^2\big\} \dif^3 x 
\end{align*}
where $c>0$ is a coupling constant.
This energy has a topological lower bound due to Faddeev,
$$
{E}_{\texttt{Skyrme}}(\varphi) \geq 12\pi^2\abs{\mathrm{deg}(\varphi)},
$$
with equality if some first order equations, generically called Bogomol'nyi (or BPS) equations,  are satisfied. Unfortunately, the BPS equations for $E_{\texttt{Skyrme}}$ are extremely restrictive: they imply that $\vphi$ is a local homothety, so they have no solution on $\RR^3$, and this bound is never attained. To remedy this inconvenience many Skyrme-like models admitting topological bounds with less restrictive BPS equations have been proposed recently \cite{ada, adam, FZ, FS}. In this paper we will consider one of these \cite{FZ,FS}.

Exploiting the standard identification $\Ss^3\equiv SU(2)$, one may \cite{riv} rewrite the Skyrme energy in terms of the
pullback, by $\vphi$, of the standard orthonormal coframe, $\{\sigma_1,\sigma_2,\sigma_3\}$, of left-invariant one-forms on $SU(2)$,
$$
{E}_{\texttt{Skyrme}}(\varphi)=\sum_{i=1}^3E_{i}(\vphi),\qquad
E_i(\vphi):=\frac{1}{2}\int \big\{c\abs{\varphi^* \sigma_i}^2 + c^{-1}\abs{\dif \varphi^* \sigma_i}^2\big\} \dif^3 x.
$$
In this way, the energy is naturally decomposed into a sum of three energies, each of which
satisfies the bound $E_i(\vphi)\geq 4\pi^2|\deg(\vphi)|$,
with equality if and only if 
\beq\label{melwal}
\ast \dif \varphi^* \sigma_i = c \varphi^* \sigma_i.
\eeq
So $\vphi$ attains the Faddeev bound if and only if it attains all three of these bounds, which
is an impossibly restrictive condition (unless physical space $\R^3$ is replaced by a suitably scaled three-sphere and $\deg(\vphi)=\pm1$). 

In \cite{FZ, FS} the authors propose the study of the truncation of the Skyrme model formed by just one of these three energies, without loss of generality, $E_{\texttt{FZ}}=E_3$ say 
(with possibly non-constant coupling $c$).
Since the energy bound is attained by solutions of the single BPS equation
\eqref{melwal}, the hope is that this model will
admit BPS energy minimizers in every degree class. In fact, as Ferreira and Zakrzewski 
\cite{FZ} point out, \eqref{melwal} has no nontrivial global finite energy solutions on $\R^3$ at all.  So the energy bound for $E_3$ is unattainable on $\R^3$.  It may still be sharp, however, that is, the infimum of $E_3$ on each degree class may be $4\pi^2|\deg(\phi)|$. If so, one would expect the binding energies of skyrmions in the one-parameter family of models with
$E=\alpha(E_1+E_2)+ E_3$ to converge to $0$ as $\alpha\ra 0^+$, yielding (yet another) family of so-called {\em near BPS} Skyrme models. Numerical evidence in favour of this conjecture has recently been obtained \cite{patspe}. 
Furthermore, the existence question for $E_3$ on a more general Riemannian domain $(M^3,g)$ remains open.

In this paper the key idea is to interpret the one-form $\eta=\sigma_3$ as a {\em contact form} on the target space $\Ss^3$. This allows us to define a natural generalization of $E_{\texttt{FZ}}$ to the class of maps $\vphi:(M^3,g)\ra (N^3,\eta)$ from an
oriented Riemannian 3-manifold $M$ to a contact 3-manifold $N$. This energy still has a topological lower bound, attained by solutions of \eqref{melwal} (with $\sigma_i$ replaced by $\eta$) that motivates its study in its own right.

The rest of the paper is structured as follows. In section \ref{sec2} we introduce the generalized FZ model
and review the basic notions from contact geometry required in its formulation. We derive its energy bound, compute its Euler-Lagrange equations and establish some properties of its BPS solutions. Special attention is paid to the cases when the target manifold is a circle bundle over a surface and has a Sasakian metric, respectively.  In section \ref{sec3} we restrict to the case of spheres ($M=N=\Ss^3$) and show that for the coupling constant $c=2$ any solution is homotopic to the identity. Finally, in section \ref{sec4} we briefly consider the case of spatially dependent coupling $c$ (as advocated in \cite{FS}), establishing nonexistence of BPS solutions on $M=\R^3$ for couplings satisfying a natural constraint.

We conclude this introductory section by describing our conventions. All maps, unless otherwise stated, are assumed to be (at least) $C^2$. On a Riemannian manifold $(M,g)$ we denote by $\flat:T_xM\ra T^*_xM$ the  isomorphism $X^\flat(Y):=g(X,Y)$ and by $\sharp$ its inverse. The Hodge isomorphism from $p$-forms to $3-p$ forms (we consider only 3-manifolds) is denoted $\ast$, and 
$\delta=(-1)^p\ast\d\ast$ denotes the
coderivative (adjoint to the exterior derivative $\d$ which on 1-forms reads $\d \eta (X,Y)=X(\eta(Y))-Y(\eta(X))- \eta([X,Y])$ \footnote{Note that this differs from \cite{blair} where a factor of 1/2 multiplies the right hand side.}). Note that for any vector field $X$, 
$\div X=-\delta X^\flat$ and $\curl X=\big(\ast\d X^\flat\big)^\sharp$.
 The Laplacian $\Delta$ will always mean the
Hodge Laplacian $\Delta=\d\delta+\delta\d$ (so, for functions, $\Delta=-\div\circ\grad$). $S^n$ will denote the $n$-sphere as an oriented manifold (but with no assumed metric), while $\Ss^n(r)$ will denote $S^n$ with the usual round metric of radius $r$; unit radius (round) spheres will simply be denoted as $\Ss^n$.

\section{Ferreira-Zakrzewski model and contact geometry}\label{sec2}

In this section, we place the Ferreira-Zakrzewski BPS Skyrme model \cite{FZ, FS} in the context of contact geometry, compute its Euler-Lagrange equations and deduce some basic properties of its solutions.

Let $\varphi:M\ra N$ be a $C^1$ map from an oriented Riemannian 3-manifold $(M,g)$ to a closed oriented $3$-manifold $N$ equipped with a contact form $\eta$. Recall that this, by definition, is a one-form on $N$
with $\eta\wedge\d\eta$ nowhere zero. Every closed orientable $3$-manifold possesses contact structures \cite{mar}, so this places no further constraint on the choice of a target space. To any such map we associate the energy
\beq \label{energcontact}
E(\varphi)=\frac12\int_M*(c|\varphi^*\eta|^2+c^{-1}|\varphi^*\d\eta|^2)
\eeq
where $c>0$ is a coupling constant. If $M$ is euclidean $\R^3$, $c$ can be set to $1$ by a rescaling of coordinates, but otherwise it is a nontrivial parameter of the model. The Ferreira-Zakrzewski model is recovered if we take $N=S^3$ with its
usual contact form (equivalently, $N=SU(2)$ with $\eta$ any left-invariant one-form).

Since $\frac12\eta\wedge\d\eta$ is nonvanishing, we may interpret it as a volume form on $N$. Then the energy functional above admits a topological lower bound, as follows: 
\bea
E(\vphi)&=& \frac{c}{2}\int_M*|\varphi^*\eta\mp c^{-1}*\d\varphi^*\eta|^2\pm \int_M*\ip{\varphi^*\eta,*\d\varphi^*\eta}\\
&\geq&\pm \int_M\varphi^*(\eta\wedge\d\eta)=\pm 2{\rm Vol}(N)\deg(\varphi)
\eea
where $\deg(\varphi)\in\Z$ is the topological degree of $\varphi$ (so $\varphi^*[N]=\deg(\varphi)[M]\in H^3(M,\Z)$ if $M$ is compact, where $[M]$, $[N]$ denote the fundamental classes generating $H^3$). By switching the orientation on $M$ if necessary, we may, and henceforth will, assume that $\deg(\varphi)\geq 0$ and take the upper signs. So $E(\varphi)\geq 2{\rm Vol}(N)\deg(\varphi)$ with equality if and only if
\beq\label{bog}
*\d\varphi^*\eta=c\varphi^*\eta.
\eeq
Note that this, the {\em Bogomol'nyi} equation, is a first order PDE for $\varphi$, in contrast to the Euler-Lagrange equations for $E$ which are, as we will see, second order.

A vacuum solution is any map $\vphi$ for which $E(\vphi)=0$. Clearly $\vphi$ is a vacuum if and only if $\vphi^*\eta = 0$ everywhere, that is, if $\vphi$ is \textit{isotropic} (in this case $\deg(\varphi)=0$). 

Consider $\wh{\xi}\in\Gamma(TM)$, the vector field dual, with respect to $g$, to the one-form $\vphi^*\eta$
(the name $\wh\xi$ will be explained shortly). 
If $\vphi$ satisfies (\ref{bog}),
\beq\label{coclosed}
\di\wh\xi=-\delta\vphi^*\eta=*\d*\vphi^*\eta=c^{-1}*\d\d\vphi^*\eta=0
\eeq
so $\wh\xi$ is divergenceless. Furthermore, it is pointwise parallel to its own curl, $\big(*\d\vphi^*\eta\big)^\sharp$. Such a vector field is usually called a {\em Beltrami field} (actually here we have a \textit{curl eigenfield} or  \textit{strong Beltrami field} as the proportionality factor is constant). 

\begin{de}
A nonisotropic map $\vphi:(M^3,g)\ra (N^3,\eta)$ satisfying (\ref{bog}) will be called a {\em Beltrami map}, or \emph{strong Beltrami map} if $c$ is constant. To be more precise we will sometimes call such a map \emph{(strong) $c$ - Beltrami map}.
\end{de}
Note that
\beq\label{cliatsc}
\Delta\vphi^*\eta=(\delta\d+\d\delta)\vphi^*\eta=\delta\d\vphi^*\eta=*\d**c\vphi^*\eta=c^2\vphi^*\eta,
\eeq
so $\vphi^*\eta$ is an eigenform of $\Delta$ on $(M,g)$ with eigenvalue $c^2$. If $M$ is compact, the spectrum of $\Delta$ is discrete, so this places a strong constraint on the allowed values of the coupling constant $c$ if (\ref{bog}) is to have any nontrivial solutions.

\begin{re} 
If we fix a $1$-form $\alpha$ as an arbitrary $*\d$-eigenform on $M$, then finding a Beltrami map $\varphi$ reduces to the \emph{pullback equation}  \cite{dac} $\vphi^* \eta = \alpha$. 
\end{re}

To formulate the Euler-Lagrange equations efficiently we must recall some basic facts from contact geometry
(see \cite{blair} for a more detailed exposition).
The {\em contact distribution} $\DD\subset TN$ is $\ker\eta$, that is, the assignment to each $y\in N$ of the two
dimensional plane $\DD_y\subset T_yN$ on which $\eta_y:T_yN\ra\R$ vanishes. The {\em Reeb} vector field $\xi$ is the
unique vector field such that 
\beq\label{defxi}
\eta(\xi)=1, \qquad  \iota_\xi\d\eta=0.
\eeq
It is convenient to equip $N$ with a so-called {\em adapted} metric $h$ (sometimes called a  ``compatible" or ``associated" metric). Given any metric $h$, we may define 
a $(1,1)$ tensor $\phi$ on $N$ by demanding
\beq\label{phidef}
h(X,\phi Y)=\frac12\d\eta(X,Y)
\eeq
for all $X,Y\in\Gamma(TN)$. 
Note that $\phi$ is skew, and that, by \eqref{defxi}, its kernel is spanned by $\xi$. The metric $h$ is
said to be adapted to the contact structure $\eta$ if $h(\xi, \xi)=1$ and the endomorphism $\phi$ satisfies
$\d\eta(\phi X, \phi Y)=\d\eta(X, Y)$ for any $X,Y \in \Gamma(\DD)$ and 
\begin{equation}\label{phi2}
\phi^2 = -I + \xi \otimes \eta,
\end{equation}
meaning that $\phi$ acts like an almost complex structure on $\DD=\xi^\perp$. In this case, $\eta^\sharp = \xi$ and the volume form defined by $h$ coincides with $\frac12\eta\wedge\d\eta$. As $\abs{\eta}_h=1$ we can easily deduce that $\ast \d \eta = 2 \eta$.  Actually an alternative definition of adapted metric is the following \cite{ch}: $h$ is an adapted metric to $\eta$ iff $\abs{\eta}_h=1$ and $\ast \d \eta = 2 \eta$. We emphasize that given a contact form, there is an abundance of compatible metrics, obtained using the polar decomposition of non-singular skew-symmetric matrices \cite{blair}.

Given any vector field $X$ on $N$ we shall denote by $X^\DD$ its component in $\DD$, that is
$X^\DD=X-\eta(X)\xi=-\phi^2X$.

\medskip
To any section $v$ of $\varphi^{-1}TN$ (the vector bundle over $M$ whose fibre above $x\in M$ is $T_{\varphi(x)}N$) we may associate the unique vector field $\wh{v}$ on $M$ such that
\beq
g(\wh{v},Y)=h(v,\d\varphi(Y))
\eeq
for all $Y\in\Gamma(TM)$. Equivalently $\wh{v}=\big(\varphi^* v^\flat\big)^\sharp$. Note that the image of the Reeb field $\xi$ (or, rather $\xi\circ\varphi$) under this map, is
\beq
\wh\xi=(\varphi^*\eta)^\sharp,
\eeq
explaining our choice of notation for this field.

\begin{pr}[Euler-Lagrange equations]
A map $\varphi:(M,g) \to (N, h,\eta)$ is a critical point for the energy functional 
$E(\varphi)$ if and only if it satisfies the following Euler-Lagrange equations:
\begin{equation}\label{EL}
\big(\dif \varphi\left((c^2\varphi^* \eta+\delta \dif \varphi^* \eta)^\sharp \right)\big)^\DD=0, \quad \delta \varphi^* \eta =0.
\end{equation}
\end{pr}

\begin{proof} 
Let $D$ be a compact domain of $M$ and let $\{\varphi_t\}$
be a variation of $\varphi$ supported in $D$ with variation vector field $v  \in \Gamma(\varphi^{-1}TN)$:
$$
v(x) = \displaystyle{\frac{\partial \varphi_t}{\partial t}(x) \Bigl\lvert _{t=0}} \in T_{\varphi(x)}N, \qquad \forall x \in M.
$$ 
Define $\Phi : M \times (-\epsilon, \epsilon) \to N$ by $\Phi(x, t) = \varphi_t (x)$. We denote by $\nabla^\Phi$ the pull-back connection on $\Phi^{-1}TN$ (see \cite{EL} for its definition and properties). When we write simply $\nabla$ we refer to the Levi-Civita connection on the respective manifold.

We now compute the derivatives of each of the two terms in the energy functional $E(\varphi_t)$ over $D$. Analogous to \cite[Prop. 2.2]{sve}, we obtain
\begin{equation}\label{2ndterm}
\frac{d}{d t}\Bigl\lvert _{t=0} \int_D \tfrac{1}{2} \abs{\varphi_{t}^{*} \dif \eta}^2 \vol_g =
\int_D \dif \eta\left(v, \dif \varphi\left((\delta \dif \varphi^* \eta)^{\sharp}\right)\right) \vol_g.
\end{equation}

Also, with respect to a local orthonormal frame $\{e_i\}_{i=1,2,3}$ on $M$, we have
\begin{equation*}\label{1stterm}
\begin{split}
\frac{d}{d t}\Bigl\lvert _{t=0} \int_D \tfrac{1}{2}\abs{\varphi_{t}^{*} \eta}^2 \vol_g 
&= \int_D \varphi^{*} \eta(e_i) \frac{d}{d t}\Bigl\lvert _{t=0} \langle \dif \varphi_t(e_i), \xi \circ \varphi_t \rangle \vol_g\\
&= \int_D \varphi^{*} \eta(e_i) \left(\langle \nabla^\Phi_{\partial_t} \dif \varphi_t(e_i), \xi \circ \varphi_t \rangle + \langle \dif \varphi_t(e_i), \nabla^\Phi_{\partial_t}\xi \circ \varphi_t \rangle \right)\Bigl\lvert _{t=0} \vol_g\\
&= \int_D \varphi^{*} \eta(e_i) \left(\langle \nabla^\varphi_{e_i} v, \xi \circ \varphi \rangle + \langle \dif \varphi(e_i), \nabla_{v} \xi \rangle \right) \vol_g\\
&= \int_D \varphi^{*} \eta(e_i) \left(e_i(\langle v, \xi \circ \varphi \rangle)- \langle  v,  \nabla^{\varphi}_{e_i} \xi \circ \varphi \rangle + \langle \dif \varphi(e_i), \nabla_{v}\xi \circ \varphi \rangle \right) \vol_g\\
&= \int_D g(\widehat{\xi}, e_i) \left(e_i(\langle v, \xi \circ \varphi \rangle)- \langle  v,  \nabla_{\dif \varphi (e_i)} \xi \rangle + \langle \dif \varphi(e_i), \nabla_{v}\xi \circ \varphi \rangle \right) \vol_g\\
& = \int_D \big\{\widehat{\xi}(\eta(v)) + \dif \eta (v , \dif\varphi \big( \widehat \xi)\big)\}\vol_g\\
& = \int_D \big\{-\eta(v)\di \widehat \xi  + \dif \eta \big(v , \dif\varphi(\widehat \xi)\big)\}\vol_g,
\end{split}
\end{equation*}
where in the last equality we have used the divergence theorem and the fact that $v$ has support in $D$. Combining with \eqref{2ndterm} we get
\begin{equation*}
\begin{split}
\frac{d}{d t}\Bigl\lvert _{t=0} E(\varphi_t; D) = & \int_D 
\big\{-c\eta(v)\di (\varphi^* \eta)^\sharp  + 
\dif \eta \big(v ,  \dif\varphi\left(c(\varphi^* \eta)^\sharp + c^{-1} (\delta \dif \varphi^* \eta)^\sharp\right)\}\vol_g\\
= &c\int_D \left\langle v ,  - (\di (\varphi^* \eta)^{\sharp})\xi + 2\phi \dif\varphi\left((\varphi^* \eta)^{\sharp} + c^{-2} (\delta \dif \varphi^* \eta)^{\sharp}\right)\right\rangle\vol_g,
\end{split}
\end{equation*}
where we used the property  \eqref{phidef}. As $\phi X \perp \xi$ for any tangent vector $X$, and since $v$ and $D$ were arbitrary, we conclude that $\varphi$ is a critical point of $E$ if and only if the claimed Euler-Lagrange equations are satisfied.
\end{proof}

Assume $\varphi$ is strong Beltrami (satisfies \eqref{bog}). Then as we have already shown, $\vphi^*\eta$ is coclosed (see (\ref{coclosed})), so automatically satisfies the second EL equation. Since $\vphi$ minimizes $E$ in its homotopy class, we expect it to satisfy the first EL equation also. As we have already observed (\ref{cliatsc}), $\delta\d\vphi^*\eta=
c^2\vphi^*\eta$, so $\vphi$ satisfies the EL equations if and only if $\left(\d\vphi((\vphi^*\eta)^\sharp)\right)^\DD=0$, that is, if and only if $\d\vphi(\wh\xi)$ is pointwise parallel to $\xi\circ\varphi$. We will shortly verify this expectation without appeal to variational methods (see Proposition \ref{prop2} $(iii)$).

For each $x\in M$, denote by $\d\vphi^T:T_{\vphi(x)}N\ra T_xM$ the adjoint of $\d\vphi:T_xM\ra T_{\vphi(x)}N$ (the unique linear map with $g(X,\d\vphi^T(Y))\equiv h(\d\vphi(X),Y)$). The {\em Cauchy-Green tensor} of $\vphi$,
$\d\vphi^T\circ\d\vphi$, is a non-negative, self-adjoint $(1,1)$ tensor field on $M$ whose eigenvalues $\lambda_1^2,\lambda_2^2,\lambda_3^2$ have proven useful in the study of Skyrme models (in the Skyrme literature,
$\d\vphi^T\circ\d\vphi$ is usally denoted ${D}$ and called the {\em strain tensor} \cite{mant}). Recall that
$x\in M$ is a critical point of $\vphi$ if $\d\vphi_x$ has nonmaximal rank or, equivalently, if at least one eigenvalue of $(\d\vphi^T\circ\d\vphi)_x$ vanishes.

\begin{pr}[Properties of strong Beltrami maps]\label{prop2}
Let $\varphi: (M,g) \to (N,h,\eta)$ be a strong $c$ - Beltrami map. Then:
\begin{itemize}
\item[$(i)$] The critical points of $\varphi$ coincide with the zeros of the $1$-form $\varphi^* \eta$. In particular, the critical set of $\vphi$ is nowhere dense.

\item[$(ii)$]  The plane field defined by the equation $\varphi^* \eta=0$ is a confoliation. 

\item[$(iii)$] $\d\vphi(\wh{\xi})$ is pointwise collinear with $\xi\circ\vphi$ with proportionality factor $|\vphi^*\eta|^2$.

\item[$(iv)$] $\wh\xi$ is an eigenvector of $\d\vphi^T\circ\d\vphi$ with eigenvalue $\lambda_1^2=|\vphi^*\eta|^2$.

\item[$(v)$] The eigenvalues  of $\d\vphi^T\circ\d\vphi$ satisfy $4\lambda_2^2\lambda_3^2=c^2\lambda_1^2$.

\end{itemize}
\end{pr}

\begin{proof} 
$(i)$ $x\in M$ is a critical point of $\vphi$ if and only if $\vphi^*\Omega=0$ at $x$ where $\Omega$ is any nondegenerate $3$-form on $N$. This holds, in particular, for the choice $\Omega=\eta\wedge\d\eta$. But, by (\ref{bog}),
\beq\label{useful}
\vphi^*(\eta\wedge\d\eta)=\vphi^*\eta\wedge c*\vphi^*\eta=c|\vphi^*\eta|^2\vol_g
\eeq
so $x$ is critical precisely when $(\vphi^*\eta)_x=0$.

It follows that the critical set of $\vphi$ is the zero set of the Beltrami field $\wh\xi$. But since any Beltrami field automatically satisfies the elliptic equation obtained by taking the $\curl$ of the defining equation, it has the unique continuation property \cite{aro} so its zero set is nowhere dense (or otherwise it is identically zero, which is excluded, since Beltrami maps are nonisotropic by definition).
 
 $(ii)$ From \eqref{useful} we have, in particular, $\vphi^*\eta\wedge\d\vphi^*\eta\geqslant 0$ so that the codimension one distribution (plane field) given by kernel of the 1-form $\vphi^*\eta$ on $M$ defines a confoliation (see \cite{ET} for the definition).

$(iii)$ At any critical point $\wh\xi(x)=0$ so $\d\vphi_x(\wh\xi)=0$ and the conclusion holds trivially.
At a point where $\varphi$ has maximal rank, any $v \in \Gamma(\varphi^{-1}TN)$ can be written as $v= \dif \varphi(Y)$ for some vector field $Y$. Then for such an arbitrary $v$ we have 
\bea
\d\eta(v,\d\vphi(\wh\xi))=\vphi^*\d\eta(Y,\wh\xi)
=c(*\vphi^*\eta)(Y,\wh\xi)
&=&c(*\vphi^*\eta)(Y,(\vphi^*\eta)^\sharp)\nonumber\\
&=&c*(\vphi^*\eta\wedge\vphi^*\eta)(Y)=0
\eea
 where we used the general identity for  1-forms $\theta,\beta$,
$\imath_{\theta^\sharp}\ast \beta \equiv \ast(\theta \wedge \beta)$. By using \eqref{defxi}, it follows that $\dif\varphi(\wh\xi)$ must be collinear with $\xi \circ \varphi$ at every $x$ where $\d\vphi$ has maximal rank.
Thus $\d\vphi(\wh\xi)=\mu\xi\circ\vphi$ for some function $\mu:M\ra\R$. Since $\xi$ has unit length,
\bea
\mu=h(\d\vphi(\wh\xi),\xi\circ\vphi)
=g(\wh\xi,\d\vphi^T(\xi\circ\vphi))
=g(\wh\xi,\wh\xi)\label{caslai}
=|\vphi^*\eta|^2.
\eea

$(iv)$ By $(iii)$ $\d\vphi(\wh\xi)=\mu\xi\circ\vphi$ where $\mu = |\vphi^*\eta|^2$. Applying $\d\vphi^T$ to this, and recalling that $\wh{\xi}=\d\vphi^T(\xi\circ\phi)$, $(\d\vphi^T\circ\d\vphi)(\wh\xi)=\mu\wh\xi$, so $\wh\xi$ is an eigenvector of the Cauchy-Green tensor with eigenvalue $\mu$. 

$(v)$ For any map $\vphi: M \ra N$, $|\vphi^*\vol_h|^2=\det(\d\vphi^T\circ\d\vphi)=\lambda_1^2\lambda_2^2\lambda_3^2$. But $\vol_h=\frac12\eta\wedge\d\eta$, so, by (\ref{useful}),
\beq\label{weaker}
\lambda_1^2\lambda_2^2\lambda_3^2=\frac14c^2|\vphi^*\eta|^4=\frac14c^2\lambda_1^4.
\eeq
Hence, away from the critical set, $4\lambda_2^2\lambda_3^2=c^2\lambda_1^2$. But the critical set is nowhere dense, and $\lambda_i^2:M\ra\R$ are continuous, so this equation holds globally.
\end{proof}

We remark that, by $(i),(iv)$ and $(v)$ of the above and by using \cite{stern}, we have:
\begin{cor}
$(i)$ A strong Beltrami map can have no points where its differential has rank $2$. 

\noindent $(ii)$ A strong Beltrami map has non-negative Jacobian (determinant) at each point of $M$. In particular, it is surjective if $M$ is compact. 
\end{cor}

Let us focus now on the zero set of strong Beltrami fields, which yields the critical set of strong Beltrami maps. First notice that, according to \cite{generic}, for a generic Riemannian metric the zero set of a strong Beltrami field consists of isolated points. Beside this generic situation, only 1-dimensional zero sets (and critical sets of the associated Beltrami map) are allowed in the analytic setting, as shown in the following

\begin{pr}
If a Riemannian $3$-manifold is analytic (i.e. both the metric and the differential structure belong to $C^\omega$),  then the zero set of a strong Beltrami field on that manifold must have codimension $\geqslant 2$. 
\end{pr}

\begin{proof} 
Such a field is necessarily analytic, so this follows from the more detailed characterization of its zero set established by Gerner \cite{gerner}.
\end{proof}

\paragraph{Boothby-Wang targets.}
Let us recall that the symplectic Dirichlet energy of a map $\phi$ taking values in a symplectic manifold $(S, \Omega)$ is $F(\phi)=\tfrac{1}{2}\norm{\phi^* \Omega}_{L^2}^2$. This energy was introduced in \cite{sve, sve1} and its critical points are also called $\sigma_2$-\emph{critical maps} \cite{slo}. When $S= \Ss^2(1/2)$, the functional $F$ represents the strong coupling limit of the static Skyrme-Faddeev energy \cite{fad}.
 
The following proposition reveals an important connexion between Beltrami and $\sigma_2$-critical maps in the case when the codomain  is a compact {\em regular} contact manifold. Regular, in this context, means that the Reeb field $\xi$ is a regular vector field: every point of the manifold has a neighbourhood through which every (necessarily closed) integral curve of $\xi$ passes at most once. By \cite[$\S 3.3$]{blair} any such manifolds is the total space of a circle bundle over a surface $\Sigma$ of integral class and the contact form yields a connexion form with curvature given by the symplectic form on $\Sigma$: $\pi^*\Omega=\tfrac{1}{2}\dif \eta$. We usually refer to these manifolds as \textit{Boothby-Wang fibrations}. A typical example is the Hopf fibration of the 3-sphere.

\begin{pr}\label{BW}
Let $(N, \eta)$ be a Boothby-Wang fibration over the Riemannian 2-manifold $\Sigma$ endowed with the area 2-form $\Omega$. Denote by $\pi : N \to \Sigma$ the bundle projection. If $\varphi:(M, g) \to (N, \eta)$ is a strong Beltrami map, then the map $\widetilde{\varphi}=\pi \circ \varphi: (M, g) \to (\Sigma, \Omega)$ is a $\sigma_2$-critical map. 
\end{pr}

\begin{proof}
According to \cite[Corollary 2.4.]{sve} $\widetilde{\varphi}$ is a $\sigma_2$-critical map if and only if
 $\dif \widetilde{\varphi} \big( (\delta\widetilde{\varphi}^*\Omega)^\sharp\big)=0$. By the fundamental property of Boothby-Wang fibrations, $\pi^*\Omega=\tfrac{1}{2}\dif \eta$, so in our case 
 $\dif \widetilde{\varphi} \big( (\delta\widetilde{\varphi}^*\Omega)^\sharp\big)=\tfrac{1}{2}
 \dif \pi\big(\dif \varphi (( \delta\dif \varphi^*\eta)^\sharp)\big)$. As $\varphi$ is a strong Beltrami map, we can use \eqref{cliatsc} then Proposition \ref{prop2} $(iii)$, to obtain
 $\dif \widetilde{\varphi} \big( (\delta\widetilde{\varphi}^*\Omega)^\sharp\big)=
\tfrac{c^2}{2} \dif \pi\big(\dif \varphi ((\varphi^*\eta)^\sharp)\big)= 
\tfrac{c^2}{2} \dif \pi\big(\dif \varphi (\widehat{\xi})\big)= \tfrac{c^2}{2}  \abs{\widehat{\xi}}^2\dif \pi(\xi)=0$, where in the last equality we used the fact that the Reeb field is tangent to the fibres of the Boothby-Wang fibration. 
\end{proof}

Combining this result with Proposition \ref{prop2} $(iii)$  we deduce that the Beltrami field $\widehat \xi$ on $(M,g)$ associated with a strong Beltrami map onto a Boothby-Wang fibration (such as e.g. $\Ss^3$) has only closed orbits. This is a very particular dynamical feature that shows us the strength of the Beltrami map condition in this context.

\paragraph{Sasakian targets.}
A contact manifold $(N, h, \eta, \xi)$ is {\em Sasakian} if its Reeb field $\xi$ is a Killing vector field for the adapted metric $h$. A well known example in dimension $3$ is $\Ss^3$ with its standard contact-metric structure (for which $\xi$ is the Hopf field). Actually there are three diffeomorphism classes $\Ss^3/\Gamma$, $\mathrm{Nil}^3/\Gamma$, and $\widetilde{SL}(2, \RR)/\Gamma$ supporting a compact Sasakian 3-manifold structure \cite{gei} (here $\Gamma$ is any discrete subgroup of the isometry group with respect to the canonical metric). The companion classification of the Sasakian metrics \cite{bel1, bel2} reveals two families of metric deformations, including the notable example of the weighted Sasakian metric on $S^3$.

Recall that, given a map $\varphi: (M,g) \to (N,h)$, the section of the the pullback bundle $\varphi^{-1}TN$ given as $\tau(\varphi)=\nabla_{e_i}^{\varphi}\dif \varphi(e_i)-\dif \varphi(\nabla_{e_i}e_i)$ is called the \textit{tension field} of $\varphi$, where $\{e_i \}$ is an orthonormal frame. A map with zero tension field is a \textit{harmonic map} \cite{ES}. The following proposition shows that the tension field of a strong Beltrami map is in (the pullback of) the contact distribution of the codomain, provided that this is Sasakian.

\begin{pr}[Tension field of a Beltrami map]\label{harmonic}
Let $\varphi: (M^3,g) \to (N^3,h,\eta,\xi)$ be strong Beltrami and $N$ be Sasakian. 
Then  
\begin{align}
&\langle \tau(\varphi), \xi \circ \varphi \rangle =0 \\
&\ip{\phi\tau(\vphi)+\left(|\d\vphi|^2 - \abs{\varphi^* \eta}^2 - \tfrac{c^2}{2}\right)\xi \circ \varphi, \ \d\vphi(X)} - \sum_i\ip{\d\vphi(e_i), \ \phi \nabla d\vphi(e_i, X)}=0. \label{phi_tens}
\end{align}
\end{pr}

\begin{proof} 
Since $\vphi$ is strong Beltrami, $\vphi^*\eta$ is coclosed.
But in any orthonormal frame $\{e_i \}$ for $(M,g)$ (with summation over the repeated index $i$ understood),
\begin{equation}
\begin{split}
\delta \varphi^* \eta    & = \ \left(\nabla_{e_i}  \varphi^* \eta\right)(e_i)\\
& = \ e_i\left[\eta(\dif \varphi(e_i))\right]-\eta(\dif \varphi(\nabla_{e_i} e_i))\\
& = \ e_i[\langle \dif \varphi(e_i), \xi \circ \varphi \rangle]-\langle\dif \varphi(\nabla_{e_i} e_i), \xi \circ \varphi \rangle\\
& = \ \langle \tau(\varphi), \xi \circ \varphi \rangle+\langle\dif \varphi(e_i), \nabla_{e_i}^{\varphi}\xi \circ \varphi \rangle= \ \langle \tau(\varphi), \xi \circ \varphi \rangle ,
\end{split}
\end{equation}
where in the last equality we used the fact that $\xi$ is Killing for $(N,h)$, so $h(Y, \nabla_{Y}\xi)=0$ for any $Y$. Therefore the condition $\delta \varphi^* \eta=0$ is \textit{equivalent} with $\tau(\varphi)\perp \xi$.

From \eqref{bog} it follows also that $\delta \varphi^* \dif \eta={c^2} \, \varphi^* \eta$. Again summing over a local orthonormal frame $\{e_i\}$ on $M$, we have, by virtue of \eqref{phidef}, for all $X\in \Gamma(TM)$,
\begin{equation}\label{cala}
\begin{split}
(\delta \varphi^* \dif \eta)(X)    & = - \left(\nabla_{e_i}  \varphi^* \dif \eta\right)(e_i, X)\\
& = -e_i[\d\eta(\d\vphi(e_i)),\d\vphi(X))]+\vphi^*\d\eta(\nabla_{e_i} e_i, X) +\vphi^*\d\eta(e_i, \nabla_{e_i}X)\\
& =2\{ - e_i\left[\langle \dif \varphi(e_i), \phi \dif \varphi(X)\rangle\right]
+\langle \dif \varphi(\nabla_{e_i}e_i), \phi \dif \varphi(X)\rangle
+\langle \dif \varphi(e_i), \phi \dif \varphi(\nabla_{e_i}X)\rangle \}\\
& = -2\ip{\nabla\d\vphi(e_i, e_i), \phi\d\vphi(X)} - 2\ip{\d\vphi(e_i),\nabla_{e_i}^\vphi(\phi\d\vphi(X))-\phi\d\vphi(\nabla_{e_i}X)}\\
& = 2\ip{\phi\tau(\vphi),\d\vphi(X)}-2\ip{\d\vphi(e_i),\nabla_{e_i}^\vphi(\phi\d\vphi(X))-\phi\d\vphi(\nabla_{e_i}X)}
\end{split}
\end{equation}
since $\phi$ is skew.  To proceed further, recall the following identity that holds true on any Sasakian manifold \cite[Theorem 6.3]{blair}:
\beq\label{phinabla}
(\nabla_X \phi) Y=h(X,Y)\xi - \eta(Y)X.
\eeq
This immediately gives us
\beq\label{cl}
\nabla_{e_i}^\vphi(\phi\d\vphi(X))-\phi\d\vphi(\nabla_{e_i}X)
=-\vphi^*\eta(X)\d\vphi(e_i)+h(\d\vphi(e_i),\d\vphi(X))\xi \circ \vphi +\phi \nabla d\varphi(e_i, X).
\eeq
Substituting \eqref{cl} into \eqref{cala}, one sees that
\bea
(\delta \varphi^* \dif \eta)(X) &=&2\ip{\phi\tau(\vphi),\d\vphi(X)}
+2\ip{\d\vphi(e_i),\vphi^*\eta(X)\d\vphi(e_i)-h(\d\vphi(e_i),\d\vphi(X))\xi\circ\vphi}\nonumber \\
& \ & - 2\langle d\varphi(e_i), \phi \nabla d\varphi(e_i, X) \rangle \nonumber \\
&=&2\ip{\phi\tau(\vphi)+|\d\vphi|^2\xi\circ\vphi-\d\vphi(\wh\xi), \ \d\vphi(X)}- 2\langle d\varphi(e_i), \phi \nabla d\varphi(e_i, X) \rangle. \nonumber 
\eea

\noindent Using $\d\vphi(\wh\xi)=\abs{\wh\xi}^2 \xi \circ \varphi$, cf. Prop. \ref{prop2}, we see that $\delta \varphi^* \dif \eta={c^2}\varphi^* \eta$ is equivalent to \eqref{phi_tens}.
\end{proof}

Let $\{e_1, e_2, e_3\}$ be a local orthonormal basis on $M$ given by the eigenvector fields of $\d\vphi^T\circ\d\vphi$. According to Proposition \ref{prop2} we may assume that $e_1= \frac{1}{\abs{\hat \xi}}\widehat{\xi}$ and its corresponding eigenvalue is $\lambda_1^2 = \abs{\varphi^* \eta}^2= \abs{\widehat \xi}^2$. Since $d\varphi(\widehat \xi) = \abs{\widehat \xi}^2 \xi \circ \varphi$  and $h(d\varphi(e_i), d\varphi(e_j))=0$ if $i\neq j$ we deduce that 
$d\varphi(e_i) \perp \xi \circ \varphi$, i.e. $d\varphi(e_i) \in \ker \eta$, $i=2,3$, so that: $\phi d\varphi(e_2)= (\lambda_2/\lambda_3) d\varphi(e_3)$ and $\phi d\varphi(e_3)= -(\lambda_3/\lambda_2) d\varphi(e_2)$ at any point where $\lambda_2^2+\lambda_3^2$ does not vanish. At this point we can employ \cite[Lemma 1]{LS} and the above Proposition in order to compute the components $A$ and $B$ of the tension field $\tau(\varphi)= A d\varphi(e_2) + B d\varphi(e_3)$ of the strong Beltrami map $\varphi$.

\section{Beltrami maps between spheres}\label{sec3}

Let us start with an immediate consequence of Proposition \ref{BW}.
\begin{cor}
Let $\pi: \Ss^3 \to \Ss^2(1/2)$ denote the Hopf map and $\eta$ the standard contact form on $\Ss^3$. If $\varphi: (\Ss^3, can) \to (\Ss^3, \eta)$ is a strong $c$ - Beltrami map, then $\pi \circ \varphi$ is a critical point of the strongly coupled Faddeev-Skyrme model with Hopf index $Q(\pi\circ \varphi)=\deg \varphi$. Moreover, $\pi \circ \varphi$ is an absolute minimizer in its homotopy class of the strongly coupled Faddeev-Skyrme energy if and only if $c=2$. 
\end{cor}

Now let us recall that the spectrum of curl and the strong Beltrami fields on the unit 3-sphere $\Ss^3$ are explicitly known \cite[Theorem 5.2]{bar}, cf. also \cite{P-SS}. We have 

\begin{pr}
On  $\Ss^3$ consider the (global) standard left-invariant frame 
\begin{equation} \label{stdframe}
\begin{split}
\xi     & = - y_1 \partial_{x_1} + x_1 \partial_{y_1} - y_2 \partial_{x_2} + x_2 \partial_{y_2}\,,\\
X_1 & =-x_2 \partial_{x_1} + y_2 \partial_{y_1} + x_1 \partial_{x_2} - y_1 \partial_{y_2}\,,\\
X_2 & =-y_2 \partial_{x_1} - x_2 \partial_{y_1} + y_1 \partial_{x_2} + x_1 \partial_{y_2}\,.
\end{split}
\end{equation}
A vector field $X=f\xi + f_1 X_1+ f_2 X_2$ is strong Beltrami (i.e. curl $\mu$-eigenvector)
iff its components satisfy the system of equations
 \begin{equation} \label{curl123}
\begin{split}
(\mu-2) f    = & \, X_1(f_2) - X_2(f_1),\\
(\mu-2) f_1  = & \, X_2(f)-\xi(f_2) , \\
(\mu-2) f_2  = & \, \xi(f_1) - X_1(f),
\end{split}
\end{equation}
which implies also that $\xi(f) + X_1(f_1) + X_2(f_2)=0$.

\noindent The eigenvalues $\mu$ of the curl operator on $\Ss^3$ are $\pm(k + 2)$ with multiplicity $(k + 1)(k + 3)$, $k\in \NN$. The components  $f, f_1$, and $f_2$ of eigenvector fields for the eigenvalue $k + 2$ are, in particular,  harmonic homogenous polynomials (in Cartesian coordinates) of degree $k$ that restrict to eigenfunctions for the scalar Laplacian on $\Ss^3$ for the eigenvalue $k(k + 2)$. 
\end{pr}

\begin{pr} \label{s3deg1}
Any non-isotropic $C^1$ solution $\varphi:\Ss^3 \to \Ss^3$ of the BPS equations \eqref{bog} for $c = 2$ is homotopic to the identity, i.e., $\deg(\varphi)=1$.
\end{pr}

\begin{proof}
The BPS equation \eqref{bog} for $c = 2$ is simply the eigenform equation $\ast \dif \widehat{\eta} = 2 \widehat{\eta}$ for $\widehat \eta := \varphi^*\eta$. We saw that on $\Ss^3$, $\mu=2$ is the lowest positive  eigenvalue of the curl operator $\ast \dif$ and it has multiplicity 3. The corresponding eigenspace $\mathcal{E}(\ast \dif, \mu=2)$ has the ($L^2$- and $g$-)orthogonal basis given by $\{\eta^M, \omega_1^M, \omega_2^M\}$. If equation \eqref{bog} has a solution $\varphi$, then $\widehat \eta=\varphi^*\eta \in \mathcal{E}(\ast \dif, \mu=2)$, so it can be expanded as  $\widehat \eta=c_0\eta^M + c_1 \omega_1^M + c_2 \omega_2^M$ with (constant) `Fourier coefficients' $c_0, c_1$ and $c_2$. 
In particular, $\abs{\widehat\eta}^2=c_0^2+c_1^2+ c_2^2$ is constant, which implies that 
$\widehat\eta$ has no zeros, unless it is identically zero and $\varphi$ is isotropic (a case excluded by hypothesis). By Proposition \ref{prop2}, $\varphi$ has no critical points, so it is a local diffeomorphism; but $\Ss^3$ is simply connected, so $\varphi$ is a diffeomorphism and $\deg(\varphi)=1$. 
\end{proof}

In particular this shows that the mapping between spheres found in \cite{FZ} is not $C^1$ smooth, so cannot represent a classical solution of the model. This can be also checked directly by translating this mapping into Cartesian coordinates.

Since for any strong $c$ - Beltrami map we have 
$$
\int_M \abs{\varphi^* \eta}^2 \mathrm{vol}_g = \tfrac{2}{c}\mathrm{Vol}(N) \deg(\varphi),
$$
for a strong 2-Beltrami map $\varphi$ as in Proposition \ref{s3deg1} we must have $\lambda_1^2 = \abs{\varphi^* \eta}^2 = 1$, and Proposition \ref{prop2}$(v)$ implies that 
$\varphi$ is actually a volume preserving diffeomorphism.

Notice finally that the crucial property needed in Proposition \ref{s3deg1} is the constant norm of the associated Beltrami field $\widehat \xi$. This property is enjoyed only by the eigenvector fields of the lowest eigenvalues  $c = \pm 2$, cf. \cite[Proposition 1.5]{pss}.

\medskip

Some natural (open) questions arise here: $(i)$ If on the codomain $N=S^3$ we consider the (non-standard) contact forms explicitly given in \cite{pss}, is the situation less rigid so that we can find Beltrami maps? $(ii)$ Do other Sasakian compact codomains allow Beltrami maps? $(iii)$ Can we find (high degree) Beltrami maps from the 3-torus to $\Ss^3$, akin to Skyrmion Crystals \cite{harleaspe}? One should note that such maps, if they exist, satisfy a topological lower energy bound, so certainly minimize energy with respect to variations of the period lattice defining the torus, as well as variation of the field -- the crucial property required for them to be interpreted as soliton crystals \cite{spe}. Finally the most interesting for us is the following

\begin{con} If $M=\Ss^3$ and $N=S^3$ with the standard contact structure, there are no smooth absolute minima of the energy \eqref{energcontact} of degree $>1$, irrespective of the coupling constant.
\end{con}

This conjecture is motivated as follows. The critical set of a strong Beltrami map is the zero set of its associated Beltrami field which, for compact domains, is generically finite. But no smooth map $\varphi:S^3\ra S^3$ of degree $k>1$ can have finite critical set for topological reasons: denoting the critical set $\Sigma$, such a map would define a $k$-fold cover $S^3\less\phi^{-1}(\phi(\Sigma))\ra S^3\less\phi(\Sigma)$, which is impossible since $S^3\less\phi(\Sigma)$ is simply connected.

\section{Nonconstant coupling}\label{sec4}

The Bogomol'nyi argument holds equally well if we allow the coupling $c$ to be a positive continuous {\em function} on $M$, rather than  a constant. This possibility was explored (for $N=\Ss^3$) in \cite{FS}.  Note that, in this more general setting, the Bogomol'nyi equation implies
\beq\label{ag}
\di\wh\xi=-\delta\vphi^*\eta=-\frac{1}{c}\wh{\xi}(c),
\eeq
which, in general, does not vanish. Hence the vector field $\wh\xi$ is no longer necessarily divergenceless. 

\begin{pr}
If $\vphi:(M^3,g)\ra (N^3,\eta)$ is a $c$-Beltrami map for some positive function $c$, then $\vphi:(M^3,c^2 g)\ra (N^3,\eta)$ is a $1$-Beltrami map. In particular, the associated Beltrami field on $(M^3,g)$ has a nowhere dense zero set.
\end{pr}

\begin{proof}
Under the conformal change of metric $g\mapsto \tilde g = c^2 g$, the Hodge isomorphism on 2-forms transforms as $\ast\mapsto\tilde\ast=\frac{1}{c}\ast$, and the result follows immediately from \eqref{bog}.
\end{proof}

An immediate consequence of the above is the fact that Proposition \ref{prop2} remains valid in the context of Beltrami maps with non-constant $c>0$. 

We can construct maps $\Ss^3\ra\Ss^3$ that satisfy \eqref{bog} for an appropriately chosen coupling $c:\Ss^3\ra(0,\infty)$. A simple example is the (degree 1) suspension map defined in \cite{ES} by:
$$
\vphi_a:(\cos s,  \nvec\sin s) \mapsto (\cos \alpha(s), \nvec\sin \alpha(s)); \quad \alpha(s)=2\arctan\left[a \tan(s/2)\right], 
$$
where $a>0$ is a constant parameter and $(s,\nvec)\in [0,\pi]\times \Ss^2$. One can check that $\widehat \xi$ is collinear with its own curl (i.e.\  \eqref{bog} holds) with $c=4 a/(1 + a^2 + (1 - a^2) \cos s)$. Notice however that $\widehat \xi$ is compressible ($\di \widehat \xi \neq 0$) unless $a=1$, in which case we recover the identity map. This is also the only value of $a$ for which $\vphi_a$ is harmonic. 

Clearly, if we fix a nonconstant function $c:M\ra(0,\infty)$, the model breaks the isometry group of $M$ and, in particular, the model on $M=\R^3$ has no Lorentz invariant extension to Minkowski space. The variable coupling
must be attributed to interaction of the Skyrme field with some (otherwise unspecified) background fields
\cite{FS}.
 A more appealing possibility is to choose $c=\wt{c}\circ\vphi$, that is, fix a positive function $\wt{c}:N\ra\R$ on the {\em target}, and consider the 
energy
\beq
E(\vphi)=\frac12\int_M(\wt{c}(\vphi)|\vphi^*\eta|^2+\wt{c}(\vphi)^{-1}|\vphi^*\d\eta|^2)\vol_g,
\eeq
which maintains the Bogomol'nyi argument and does permit Lorentz invariant extension. A natural constraint on
$\wt{c}$, particularly in the case of Sasakian target, where $\xi$ is a Killing vector, is that $\xi[\wt{c}]=0$, so that $\wt{c}$ is constant along the integral curves of $\xi$. In this case, \eqref{ag} and
Proposition \ref{prop2}$(iii)$ imply that, for any solution of \eqref{bog}, $\delta\vphi^*\eta=0$. Hence
$\wh{\xi}$ is a Beltrami field (and $\varphi$ a Beltrami map). From this, we immediately get a non-existence result on $M=\R^3$.

\begin{pr} Let $N^3$ be a compact contact manifold and
$\wt{c}:N\ra(0,\infty)$ be a $C^1$ function with $\xi[\wt{c}]=0$. Then there are no
$C^1$ solutions $\vphi:\R^3\ra N$ of the Bogomol'nyi equation \eqref{bog} with $c=\wt{c}\circ \vphi$
that have finite nonzero energy.
\end{pr}

\begin{proof} 
Assume, towards a contradiction, that such a solution $\vphi$ with $0<E(\vphi)<\infty$ exists.
Since $N$ is compact and $\wt{c}:N\ra(0,\infty)$ is continuous, $\wt{c}\geq C>0$, some constant, and so, likewise, $c=\wt{c}\circ \vphi\geq C$. Hence
$\|\wh{\xi}\|_{L^2}^2=\|\vphi^*\eta\|_{L^2}^2\leq 2E(\vphi)/C<\infty$. So $\wh{\xi}$ is in $L^2(\R^3)$ and, as already observed, is a Beltrami field. But according to \cite{nad}, the only such field is
$\wh{\xi}=0$, a contradiction (since this field would also have $\d\vphi^*\eta=0$, and hence $E(\vphi)=0$). 
\end{proof}
\noindent For restrictions on the existence of Beltrami fields in terms of the proportionality factor see \cite{rare}.

\bigskip

\paragraph{Acknowledgements.} We would like to thank D. Peralta-Salas for useful discussions related to Beltrami fields and their zero sets.

\medskip

\begin{small}

\end{small}

\end{document}